\documentclass[a4paper,12pt]{amsart}

\usepackage{amsmath,amsthm,amssymb}
\usepackage{fullpage}
\usepackage{mathtools}
\usepackage{color}
\usepackage{enumitem}
\usepackage{setspace}
\usepackage[colorlinks,citecolor=blue,urlcolor=blue]{hyperref}

\newtheorem{theorem}{Theorem}
\newtheorem{proposition}{Proposition}
\newtheorem{lemma}{Lemma}
\newtheorem{corollary}{Corollary}
\newtheorem*{remark}{Remark}

\newcommand{\E}{\mathbb{E}}

\newcommand{\reals}{\mathbb{R}}

\newcommand{\eps}{\varepsilon}

\DeclareMathOperator{\supp}{supp}

\DeclareMathOperator{\dist}{dist}

\title{Reconstructing the base field from imaginary multiplicative chaos}
\author{Juhan Aru}
\author{Janne Junnila}

\begin{document}

\renewcommand{\Re}{\operatorname{Re}}
\renewcommand{\Im}{\operatorname{Im}}

\maketitle

\begin{abstract}
  We show that the imaginary multiplicative chaos $\exp(i\beta \Gamma)$ determines the gradient of the underlying field $\Gamma$ for all log-correlated Gaussian fields with covariance of the form $-\log |x-y| + g(x,y)$ with mild regularity conditions on $g$, for all $d \geq 2$ and for all $\beta \in (0,\sqrt{d})$. In particular, we show that the 2D continuum zero boundary Gaussian free field is measurable w.r.t. its imaginary chaos.
\end{abstract}
\onehalfspacing

\section{Introduction}

\noindent In this paper $\Gamma$ denotes a log-correlated Gaussian field in $d \geq 2$ dimensions. Log-correlated fields, and in particular the 2D Gaussian free field have recently played an important role in describing the continuum limits of statistical physics models, in the study of SLE processes and in probabilistic constructions of 2D Liouville Quantum Gravity (see e.g. lecture notes \cite{B, PW}). Given such a log-correlated field, one can define the Gaussian multiplicative chaos (GMC) related to $\Gamma$ formally as the exponential $\exp(\gamma \Gamma)$. The study of GMC has also recently attracted wide interest due to its many connections and applications in, for example, random matrix theory, for modelling turbulence, and in the probabilistic study of Liouville field theory (see e.g. \cite{RW} for a review and numerous applications). 

In the current article, we concentrate on the case where $\gamma$ is imaginary, and for clarity we write it as $\exp(i\beta \Gamma)$. Such a random complex-valued distribution is called the imaginary multiplicative chaos and in 2D one could think of it as a random vector field. As log-correlated Gaussian fields are merely random distributions, this complex exponential has to be given mathematical sense through a regularization and renormalization procedure: we define the imaginary chaos $\mu = \mu_\beta$ as the limit of $\exp(i\beta \Gamma_\eps(\cdot) - \frac{\beta^2}{2}\E \Gamma_\eps(\cdot)^2)$ in the space of distributions, for some approximation sequence $\Gamma_\eps$ of the underlying field $\Gamma$. Imaginary multiplicative chaos has been shown to describe the continuum limit XOR-Ising model, to be related to the sine-Gordon model \cite{JSW} (see also \cite{RVSG} for recent related work), and it has also been used as a tool to study level sets of the 2D GFF \cite{SSV}.

The aim of this short note is to present an explicit way to reconstruct the underlying log-correlated field $\Gamma$ from its imaginary chaos in the whole region of its definition. Heuristically, we are given the continuum log-correlated field $\Gamma$ modulo $2\pi \beta^{-1}$ and we are interested in determining $\Gamma$ from this information. Whereas this would be trivial for a continuous function $\Gamma$, it is much less clear for a random distribution. In fact, as the multiplicative chaos is defined only through a regularization and renormalization procedure, such a reconstruction is non-trivial already for the real multiplicative chaos $\exp(\gamma \Gamma)$, with $\gamma \in \reals$; still, this case has been resolved in \cite{BSS}.

An even richer setting presents itself in the discrete: in \cite{GS} the authors consider $\exp(i \beta \Gamma_n)$ for a discrete 2D GFF (or in other words $\Gamma_n \mod 2\pi \beta^{-1}$) and prove an interesting threshold phenomena: they show that there exist $\beta_c^-, \beta_c^+$ such that for $\beta < \beta_c^-$ the underlying field $\Gamma_n$ can be reconstructed from $\exp(i \beta \Gamma_n)$ with high precision, and for $\beta > \beta_c^+$ such a reconstruction does not exist\footnote{This corresponds to a Kosterlitz-Thouless type of phase transition for the 2D discrete GFF.}. In \cite{GS} bounds on $\beta_c^-$ were also conjectured, and proved conditionally on the measurability of the underlying 2D GFF w.r.t. the imaginary chaos in the continuum (see Conjecture 1 and Proposition 6.1 of \cite{GS}). This note thus provides the missing part to confirm their Conjecture 1.~\\

In the case of real multiplicative chaos of the 2D GFF, to determine the underlying field, it is roughly speaking possible to just mollify the chaos, take the logarithm and show that after centering this field converges to the underlying GFF \cite{BSS} (see e.g. \cite[Section~5]{AR} for an exposition). In fact, such a procedure provides a local reconstruction of the underlying field -- in every neighbourhood, the real chaos determines its underlying field. In the case of imaginary chaos, however, such a local reconstruction is not possible: adding the constant $2\pi \beta^{-1}$ to the field locally is an absolutely continuous operation which does not change the chaos.

Yet, our main result shows that even though $\Gamma$ cannot be recovered locally from its imaginary chaos, one can recover the gradient of $\Gamma$. Presence of any \textit{global} condition such as boundary conditions then determines the underlying field $\Gamma$. In fact, our approach is rather general and works in all dimensions $d \ge 2$ and for all log-correlated fields whose covariance kernels are of the form
\begin{equation}\label{eq:cov}
  C(x,y) \coloneqq \log \frac{1}{|x-y|} + g(x,y),
\end{equation}
where the function $g \in C^2(U \times U) \cap L^1(U \times U)$ is bounded from above.\footnote{Since our main result is local in nature, it is mainly the $C^2$-condition that matters. The $L^1$-requirement and boundedness from above are imposed just to conform to the setting in \cite{JSW} in order to easily cite results concerning existence and properties of imaginary chaos.}

\begin{theorem}[Recovery of the gradient]\label{thm:main}
  Let $\Gamma$ be a log-correlated Gaussian field defined on an open simply connected domain $U \subset \reals^d$ {\rm (}$d \ge 2${\rm )} with covariance $C$ as in \eqref{eq:cov}. Then for any $\beta \in (0,\sqrt{d})$ the field $\nabla \Gamma$ (understood as a vector of random distributions) is locally recoverable from the imaginary chaos $\mu_\beta$. 
\end{theorem}

Here, by locally recoverable we mean that
  for any $f \in C_c^\infty(U)$ and $k \in \{1,\dots,d\}$ the random variable $\langle \frac{\partial}{\partial x_k} \Gamma, f \rangle$ is measurable with respect to the random distribution $\mu$ restricted to any open neighbourhood of the support of $f$.

Now, for any smooth vector field $f \in C_c^\infty(U, \mathbb{R}^d)$, we have that $\langle \nabla \Gamma, f \rangle = - \langle \Gamma, \nabla \cdot f \rangle$. This defines $\Gamma$ up to a global additive constant. Thus, whenever the additive constant of the field is fixed, by for example prescribing boundary conditions or a zero mean condition, we can uniquely recover the whole field $\Gamma$. This holds in particular for the zero boundary Gaussian free field.

\begin{corollary}[Recovery of the field]
  Let $\Gamma$ be the zero boundary GFF in a simply connected domain $U \subset \reals^2$. Then for any $\beta \in (0,\sqrt{2})$ the field $\Gamma$ is measurable with respect to the random distribution $\mu$.
\end{corollary}

As mentioned already, by proving the assumption H1 in Proposition 6.1 of \cite{GS}, this corollary confirms the Conjecture 1 of \cite{GS} on the lower bound on $\beta_c^-$, the threshold for statistical reconstruction of the underlying discrete 2D GFF.

\begin{remark}
One should note an important difference between the real and imaginary chaos, related to the fact that imaginary chaos does not determine the field locally. Namely, contrary to the case of the real chaos, the reconstruction of the base field in the case of imaginary chaos requires a certain level of continuity from the underlying field. 

This might be best illustrated via the example of multiplicative cascades. Indeed, real multiplicative cascades on $[0,1]$ are random measures whose density is formally given by $M_\gamma(x) := \prod_{x \in I} e^{\gamma X_I}$, where $I$ ranges over all dyadic intervals containing the point $x$, $\gamma \in \reals$ is a parameter and $X_I$ are, say, i.i.d. Gaussians with $\E e^{\gamma X_I} = 1$. The real cascade measure $M_\gamma(x)$ determines the underlying (almost) log-correlated field $A(x) := \sum_{x \in I} X_I$ for every small enough $\gamma$, e.g. by applying the argument of \cite{BSS}. However, in the case of imaginary cascades, i.e. when we replace $\gamma \to i\beta$ and consider random distributions of the form $M_{i\beta}(x) = \prod_{x \in I} e^{i \beta X_I}$, this is simply not true. Namely, one can always shift a single $X_I$ by $2\pi \beta^{-1}$ without changing $M_{i\beta}$. 

Interestingly, the chaos for the 2D Gaussian free field can be seen as a sort of generalized cascade \cite{APS, SSV}, yet shifting the field as above would be felt by this cascade. In fact, to our knowledge, this difference in measurability is one of the few phenomena where there is a notable difference between multiplicative cascades and multiplicative chaos measures. It is an interesting question to see what are the sharp conditions for Theorem \ref{thm:main} to hold, or in other words, to see the minimum level of continuity that allows to reconstruct a random field $\Gamma$ from knowing it modulo $2\pi \beta^{-1}$. 
\end{remark}

The key observation to our proof is the following simple heuristic: if you formally take the gradient of the imaginary chaos $\mu := \exp(i\beta \Gamma)$, then you obtain $i\beta \nabla \Gamma \exp(i \beta \Gamma)$. To recover $\nabla \Gamma$, we need to just multiply this by $(i\beta)^{-1}\exp(-i \beta \Gamma)$. This sets our strategy, but naturally as things are not defined pointwise, one will need to pass through a regularization and renormalization argument to make things work. In reality, we will consider $\nabla \mu$, that is well defined in the distributional sense, and try to recover $\nabla \Gamma$ by multiplying $\nabla \mu$ with $\bar \mu * \varphi_\eta$ for a family of smooth mollifiers. The aim is then to show that after renormalization we recover $\nabla \Gamma$ in the limit. In fact, judicious but not straight-forward choices of both definitions and computations make a quite short $L^2$-argument possible in dimensions $d \geq 3$. However, in $d = 2$ an interesting resonance phenomenon occurs, and we will need to add further averaging to tame down its effect.

For the moment, our approach does not extend to $d = 1$, and the reconstruction of the field from the chaos in 1D remains an open question. On the other hand, it is not difficult to see that Theorem 1 cannot hold in any $d \geq 1$, if we replace $\mu$ by either its real or imaginary part, or in other words the derivative of the underlying field cannot be locally recovered from only $\cos(\beta\Gamma)$ or only $\sin(\beta\Gamma)$. 
It is an interesting question to determine what information is retained in the real and imaginary part of the chaos. It would be equally interesting to see whether the base field might possibly be measurable just w.r.t. to the angle the angle of the chaos, i.e. the collection of random variables $\mu(f)/|\mu(f)|$ with $f \in C_c^\infty(U)$. 
~\\

\noindent \textbf{Acknowledgements}~\\

\noindent We have benefited from very inspiring and useful discussions on the topic with C. Garban, E. Saksman, A. Sep\'ulveda and C. Webb. We are also very thankful to them for their helpful comments on the manuscript. Both authors are members of NCCR Swissmap.

\section{Proof of Theorem~\ref{thm:main}}

\noindent Let us start with a few preliminaries. Throughout this section we will let $\varphi$ be a smooth non-negative mollifier supported in the annulus $B(0,1) \setminus B(0,1/2)$ and denote $\varphi_\eta(x) \coloneqq \eta^{-d} \varphi(x/\eta)$ for all $\eta \in (0,1)$. We also fix a bounded simply connected domain $U \subset \reals^d$.

By a log-correlated Gaussian field $\Gamma$ on $U$ we mean a centered Gaussian random distribution with the covariance structure
\[\E \langle \Gamma, f \rangle \langle \Gamma, g \rangle = \int f(x)g(y) C(x,y) \, dx \, dy,\]
where $C$ is of the form \eqref{eq:cov} and we extend $C$ as $0$ outside of $U \times U$. Such a process $\Gamma$ is not given by a function, but one may check that it makes sense as a random distribution in any negative order Sobolev space $H^{-\varepsilon}(\reals^d)$, see e.g. \cite[Proposition~2.3]{JSW}.

From $\Gamma$ we then construct the imaginary chaos distribution $\mu$ via the limiting procedure
\[\mu(x) = \lim_{\delta \to 0} e^{i\beta (\Gamma * \varphi_\delta)(x) + \frac{\beta^2}{2} \E (\Gamma * \varphi_\delta)(x)^2},\]
where the convergence takes place in $H^s(\reals^d)$ for $s < -d/2$ and a posteriori $\mu$ will belong to $H^s(\reals^d)$ for $s < -\beta^2/2$ \cite{JSW}. Again, $\mu$ will not be a function (or even a measure), but below we will nevertheless freely use the suggestive notation $\int f(x) \mu(x) \, dx$ for $\langle \mu, f\rangle$ with $f \in C_c^\infty(U)$.

As both the log-correlated field $\Gamma$ and the imaginary chaos $\mu$ are random distributions they have distributional (partial) derivatives, and hence e.g. $\frac{\partial}{\partial x_1} \Gamma$ can be seen as a distribution-valued centred Gaussian process. The covariance structure of each partial derivative is determined by, say,
\[ \E \langle \partial_1 \Gamma, f \rangle^2 = \int \partial_1 f(x) \partial_1 f(y) C(x,y) \, dx\, dy.\]
In the rest of this section we will first show how to recover $\nabla \Gamma$ when $d \ge 3$ and then treat the case $d = 2$, which requires some extra care.

\subsection{Measurability in dimensions three and above}\label{sec:highdimensions}

We start by fixing some notation for derivatives. We denote by $\partial \coloneqq \frac{\partial}{\partial x_k}$ the (distributional) partial derivative with respect to the $k$th coordinate, where we may without loss of generality assume that $k=1$. Below the notation $\partial_1h(x,y)$ and $\partial_2h(x,y)$ is used for functions $h(x,y)$ defined on $U \times U$ to refer to $\frac{\partial}{\partial x_1} h(x,y)$ and $\frac{\partial}{\partial y_1} h(x,y)$ respectively. 

Now, let $f \in C_c^\infty(U)$ be any fixed test function. For all $\eta \in (0, \dist(\supp f, \partial U)/2)$ we define the following random variables
\[H_\eta \coloneqq \int dx du f(x) \mu(x) \overline{\mu(u)} e^{-\beta^2 C(x,u)} \partial \varphi_\eta(x - u).\]
To make rigorous sense of this definition, one can, for example, interpret the above as applying the random distribution $(\mu \otimes \overline{\mu})(x,y)$ on $\reals^{2d}$ to the test function \[f(x)e^{-\beta^2C(x,y)}\partial\varphi_\eta(x-y).\] As $\varphi$ is supported on an annulus, this works directly when $g(x,y)$ is smooth, i.e. when $C(x,y)$ is smooth off the diagonal. In the case when $g(x,y)$ is merely $C^2(U \times U)$, one can approximate $C(x,y)$ first by kernels with $g(x,y)$ smooth, and then define $H_\eta$ as a limit in $L^2(\Omega)$, and hence in probability. This extra limiting procedure for non-smooth kernels can be done simultaneously for countable dense set of test functions $f$ and $\eta \in \mathbb{Q}$, and thus we can safely neglect its presence in what follows.

Since $H_\eta$ are measurable w.r.t. the chaos $\mu$, Theorem~\ref{thm:main} in the case $d \ge 3$ will follow from the following proposition.

\begin{proposition}\label{prop:dge3}
  For $d \ge 3$ we have $H_\eta \to -i\beta \langle \partial \Gamma, f \rangle$ in $L^2(\Omega)$.
\end{proposition}

By expanding the square in $\E |H_\eta + i\beta \langle \partial \Gamma, f \rangle|^2$, Proposition~\ref{prop:dge3} is a consequence of the following two lemmas, the first one taking care of the cross-terms and the second handling $\E |H_\eta|^2$.

\begin{lemma}\label{lemma:crossterms}
  For $d \ge 2$ we have
  \[-i\beta\lim_{\eta \to 0} \E H_\eta \langle \partial \Gamma, f \rangle = -\beta^2\E \langle \partial \Gamma, f \rangle^2.\]
\end{lemma}

\begin{lemma}\label{lemma:diagonalheta}
  For $d \ge 3$ we have $\lim_{\eta \to 0} \E |H_\eta|^2 = \beta^2 \E \langle \partial \Gamma, f \rangle^2$. When $d=2$ we still have $\limsup_{\eta \to 0} \E |H_\eta|^2 < \infty$.
\end{lemma}

\begin{proof}[Proof of Lemma~\ref{lemma:crossterms}]
  We start by noting that for any real $\gamma$ and $\delta > 0$ one has by Girsanov theorem that
  \begin{align*}
      & \E e^{\gamma \Gamma_\delta(x) - \gamma \Gamma_\delta(u) - \frac{\gamma^2}{2} \E \Gamma_\delta(x)^2 - \frac{\gamma^2}{2} \E \Gamma_\delta(u)^2} \Gamma_\delta(y) \\
      & = e^{-\gamma^2 \E \Gamma_\delta(x) \Gamma_\delta(u)} \E(\Gamma_\delta(y) + \gamma \E \Gamma_\delta(x) \Gamma_\delta(y) - \gamma \E \Gamma_\delta(u) \Gamma_\delta(y)) \\
      & = \gamma e^{-\gamma^2 \E \Gamma_\delta(x) \Gamma_\delta(u)} (\E \Gamma_\delta(x) \Gamma_\delta(y) - \E \Gamma_\delta(u) \Gamma_\delta(y)),
  \end{align*}
  for some approximation $\Gamma_\delta$ of $\Gamma$. Hence by letting $\delta \to 0$ and using the fact that the expression is analytic in $\gamma$, we can justify the formal computation
  \[\E \mu(x) \overline{\mu(u)} e^{-\beta^2 C(x,u)} = i\beta(C(x,y) - C(u,y)).\]
  Now, by definition $\langle \partial \Gamma, f \rangle = -\int \Gamma(y) \partial f(y)$, and thus we get
  \begin{equation*}
    i\beta\E H_\eta \int \Gamma(y) \partial f(y) = -\beta^2\int f(x) \partial f(y) (C(x,y) - C(u,y)) \partial \varphi_\eta(x-u) \nonumber,
  \end{equation*}
  which by integration by parts w.r.t. $u_1$ can be further written as 
  \[\beta^2\int f(x) \partial f(y) \partial_1 C(u,y) \varphi_\eta(x-u).\]
   Since $|\partial_1 C(\eta u, y)| \lesssim \frac{1}{|y - \eta u|}$, by uniform integrability we get that
  \begin{align*}
      & \lim_{\eta \to 0} \beta^2 \int f(x) \partial f(y) \partial_1 C(u,y) \varphi_\eta(x-u) = \beta^2\int f(x) \partial f(y) \partial_1 C(x,y) \\
      & = -\beta^2 \int \partial f(x) \partial f(y) C(x,y). 
  \end{align*}
  But the last expression equals exactly $ -\beta^2 \E \langle \partial \Gamma, f \rangle^2$, and hence we conclude.
\end{proof}

\begin{proof}[Proof of Lemma~\ref{lemma:diagonalheta}]
  Expanding $\E |H_\eta|^2$ we have
  \[\E |H_\eta|^2 = \int f(x)f(y) e^{\beta^2 C(x,y) + \beta^2 C(u,v) - \beta^2 C(x,v) - \beta^2 C(y,u)} \partial \varphi_\eta(x-u) \partial \varphi_\eta(y-v).\]
  By the change of variables $u \mapsto x-\eta u$, $v \mapsto y-\eta v$ we write this as
  \[\E |H_\eta|^2 = \eta^{-2} \int f(x)f(y) E(x,y,\eta u,\eta v) \partial \varphi(u) \partial \varphi(v),\]
  where
  \begin{equation}\label{eq:E}
    E(x,y,u,v) \coloneqq e^{\beta^2 C(x,y) + \beta^2 C(x-u,y-v) - \beta^2 C(x,y-v) - \beta^2 C(y,x-u)}.
  \end{equation}
  Let us consider separately the cases $|x-y| \le 10\eta$ and $|x-y| > 10\eta$.

  \medskip

  \noindent\textbf{The diagonal part $|x-y| \le 10 \eta$ is negligible:}

  \medskip

  We have
  \begin{align*}
    & \Big|\eta^{-2} \int_{|x-y| \le 10 \eta} f(x)f(y) E(x,y,\eta u,\eta v) \partial \varphi(u) \partial \varphi(v)\Big| \\
    & \lesssim \eta^{-2} \int_{|x-y| \le 10 \eta} f(x)f(y)\frac{|x-y+v|^{\beta^2} |x-y-u|^{\beta^2}}{|x-y|^{\beta^2} |x-y-u+v|^{\beta^2}} \\
    & = \eta^{d-2} \int_{|z| \le 10} f(y) f(y+\eta z) \frac{|z+v|^{\beta^2}|z-u|^{\beta^2}}{|z|^{\beta^2} |z-u+v|^{\beta^2}} \lesssim \eta^{d-2},
  \end{align*}
  and we see that the diagonal contribution vanishes when $d \ge 3$ and stays bounded when $d=2$.
  Above we did the change of variables $x = y + \eta z$ and used the fact that the $g$ term in the covariance \eqref{eq:cov} is locally bounded and the support of $f$ is compact.
                                                                                                              
  \medskip
  \noindent\textbf{The main part $|x-y| > 10\eta$:}
  \medskip

  As $|x-y| > 10\eta$ and $1/2 \le |u|, |v| \le 1$, there are no singularities or boundary contribution when we integrate by parts w.r.t. $u_1$ to get
  \[\beta^2 \eta^{-1} \int_{|x-y| > 10 \eta} f(x)f(y) E(x,y,\eta u, \eta v) (\partial_1 C(x-\eta u, y-\eta v) - \partial_1 C(x-\eta u, y)) \varphi(u) \partial \varphi(v).\]
  Note that by our constraint $|x-y| > 10\eta$ all the second derivatives of $C(x-\eta u, y- \eta v)$, $C(x-\eta u,y)$, and $C(x, y-\eta v)$ with respect to $\eta$ are bounded uniformly for $|u|,|v| \le 1$ by a constant times $1/|x-y|^2$. Thus we have that the second order difference
  \[C(x,y) + C(x-\eta u, y-\eta v) - C(x-\eta u,y) - C(x,y-\eta v) = O(\eta^2/|x-y|^2)\]
  and we conclude that $E(x,y,\eta u,\eta v) = 1 + O(\eta^2/|x-y|^2)$. Moreover (also under the constraint $|x-y| > 10\eta$) we have
  \[\partial_1 C(x-\eta u, y-\eta v) - \partial_1 C(x-\eta u, y) = O(1/|x-y|).\]
  Hence we see that the integral over to $O$-terms will be bounded in absolute value by a constant times
  \[\eta \int_{|x-y| > 10 \eta} \frac{1}{|x-y|^3} \lesssim \begin{cases} \eta \log(1/\eta), & \text{if } d = 3 \\ \eta^{d-2}, & \text{otherwise,} \end{cases}\]
  which goes to $0$ if $d \ge 3$ and stays again bounded if $d = 2$.

  It remains to consider
  \[\beta^2 \eta^{-1} \int_{|x-y| > 10 \eta} f(x)f(y) (\partial_1 C(x-\eta u, y-\eta v) - \partial_1 C(x-\eta u, y)) \varphi(u) \partial \varphi(v).\]
  The second term in the parenthesis will vanish when integrating by parts with respect to $v_1$. For the first term we may first do the change of variables $y = z + \eta v$ and then integrate by parts to get
  \[-\beta^2 \int_{|x-z-\eta v| > 10 \eta} f(x)\partial f(z + \eta v) \partial_1 C(x-\eta u, z) \varphi(u) \varphi(v),\]
  which as $\eta \to 0$ (again by uniform integrability) tends to
  \[-\beta^2 \int f(x) \partial f(z) \partial_1 C(x,z) = \beta^2 \int \partial f(x) \partial f(z) C(x,z) = \beta^2 \E \langle \partial \Gamma, f \rangle^2. \qedhere\]
\end{proof}

\subsection{Measurability in dimension two}

We saw in the proof of Lemma~\ref{lemma:diagonalheta} that when $d=2$, a bounded but non-vanishing diagonal contribution remains, killing any hope for a $L^2$-convergence in Proposition~\ref{prop:dge3}. This absence of $L^2$-convergence is an actual phenomenon and not just a limitation of the proof. To see where it stems from, it is illustrative to consider what happens when one tries to peel off the chaos from $\partial \mu(x)$ using just a martingale approximation $\overline{\mu_n(x)}$ normalized with $e^{-\beta^2 \E \Gamma_n(x)^2}$. Indeed, consider for example the zero boundary GFF in the unit square and let $\Gamma_n(x) = \sum_{k=1}^n X_k(x)$ be its Fourier expansion. We know that then $X_k$ are of the form $Y_{j,l}\sin(\pi j x_1)\sin(\pi l x_2)$ with $Y_{j,l}$ independent centred Gaussians. Set $\mu_n(x) = e^{i\beta \Gamma_n(x) + \frac{\beta^2}{2} \E \Gamma_n(x)^2}$. Then, up to ignoring some unimportant non-homogenity in space, we have
\begin{align*}
  \int f(x) \partial \mu(x) \overline{\mu_n(x)} e^{-\beta^2 \E \Gamma_n(x)^2} \, dx & \approx -i\beta \int f(x) \mu(x) \overline{\mu_n(x)} e^{-\beta^2 \E \Gamma_n(x)^2} \partial \Gamma(x) \, dx \\
  & \quad  - \int \partial f(x) \mu(x) \overline{\mu_n(x)} e^{-\beta^2 \E \Gamma_n(x)^2} \, dx.
\end{align*}
Now, the second term above tends to $-\int \partial f(x) \, dx = 0$, as it should, also when $d=2$. One would also expect the first term to tend to $-i\beta\langle \partial \Gamma,f \rangle$, but this is not true in $d = 2$ due to an interesting resonance phenomenon. Namely, if we denote by $\mu_n^T$ the chaos corresponding to $\Gamma - \Gamma_n$, the non-negligible contribution to the first term will come from the field $\mu_n^T(x)\partial \Gamma_n(x)$, where by construction the two factors are independent. Although this sequence will stay bounded in $L^2(\Omega)$, it will not converge to the main term $\partial \Gamma$, as one might naively expect. In fact, it does not converge at all in $L^2$, but rather it converges just in law to $\partial \Gamma$ plus an independent multiple of the white noise. This is basically due to the fact that $\mu_n^T$ remains rough on the scale $n$, giving rise to a resonance effect driven by frequencies near the scale $n$ in $\partial \Gamma_n$ and $\mu_n^T$. Whereas this might hint that our strategy has no hope, it is luckily not the case: namely, the resonance and the white noise contribution resulting from it, are due to only frequencies near $n$, and these resonating frequencies will be shifting with $n$. Thus, as the different frequency bands of $\Gamma$ are independent, the fields resulting from resonances are basically giving rise to independent noise fields for all $n$ that are sufficiently far apart. Thus one could hope to just average away these noise terms.\footnote{As a possible analogy, one might also think of the usage of Ces\`aro sums / Fej\'er kernel to tame down the Gibbs phenomenon at jump discontinuities.} More precisely, as $\mu_n^T$ is mean one, one could hope that if we take samples along some sequence $n_k$ which tends to $\infty$ very fast and consider the running average, a law of large numbers will enter and provide convergence of these averages to $-i\beta\int f(x) \partial \Gamma_n(x) \, dx$. We will show below that this plan of attack indeed works.

In order to implement the above strategy using our approximations $H_\eta$ from Section~\ref{sec:highdimensions}, let us define for all $N \ge 1$ the averages
\[A_N \coloneqq \frac{1}{N} \sum_{n=1}^N H_{\varepsilon_n},\]
where $\varepsilon_n = 2^{-K^n}$ for some large enough constant $K \ge 1$. We immediately see that in the formula
\[\E |A_N + i\beta \langle \partial \Gamma, f \rangle|^2 = \E |A_N|^2 - i\beta \E A_N \langle \partial \Gamma, f \rangle + i \beta \E \overline{A_N} \langle \partial \Gamma, f \rangle + \beta^2 \E \langle \partial \Gamma, f \rangle^2\]
the cross-terms tend by Lemma~\ref{lemma:crossterms} to $-\beta^2 \E \langle \partial \Gamma, f \rangle^2$ as $N \to \infty$.

Assume now that we can show that for $n > m$ we have 
$\E H_{\varepsilon_n} \overline{H_{\varepsilon_m}} = c_{n,m} + O(\varepsilon_n^\alpha\varepsilon_m^{-\beta})$ for some constants $\alpha,\beta > 0$, where $c_{n,m}$ is such that $c_{n,m} \to \beta^2 \E \langle \partial_1 \Gamma, g \rangle^2$ as $n,m \to \infty$. Then we have
  \[\E |A_N|^2 = \frac{1}{N^2} \sum_{n,m=1}^N \E H_{\varepsilon_n} \overline{H_{\varepsilon_m}} = \frac{1}{N^2} \sum_{n,m=1}^N c_{n,m} + \frac{1}{N^2} \sum_{n=1}^N \E |H_{\varepsilon_n}|^2 + \frac{1}{N^2} \sum_{1 \le m < n \le N}^N O(\varepsilon_n^{\alpha} \varepsilon_m^{-\beta}),\]
  where the first term now tends to $\beta^2 \E \langle \partial_1 \Gamma, g \rangle^2$ as $N \to \infty$, the second term goes to $0$ since the diagonal terms are bounded by Lemma~\ref{lemma:diagonalheta}, and by choosing $K$ large enough one sees that also the third term goes to $0$.

By the above discussion it is thus enough to prove the following lemma.

\begin{lemma}\label{lemma:Asmallbeta}
  Fix any $\varepsilon \in (0,1)$. For all $\eta_1 < \eta_2/100$ we have $\E H_{\eta_1} \overline{H_{\eta_2}} = c_{\eta_1,\eta_2} + O(\eta_1^{1-\varepsilon}/\eta_2),$
  where $c_{\eta_1,\eta_2} \to \beta^2 \E \langle \partial \Gamma, f \rangle^2$ as $\eta_1,\eta_2 \to 0$.
\end{lemma}

\begin{proof}
  We have after the change of variables $u \mapsto x - \eta_1 u$ and $v \mapsto y - \eta_2 v$ that
  \[\E H_{\eta_1} \overline{H_{\eta_2}} = \eta_1^{-1} \eta_2^{-1} \int f(x)f(y) E(x,y,\eta_1 u,\eta_2 v) \partial \varphi(u) \partial \varphi(v),\]
  where $E$ is given by \eqref{eq:E}. We will begin by showing that integrating close to the singularities of $E$ gives a contribution of order $\eta_1/\eta_2$.

  \medskip
  \noindent \textbf{The diagonal parts $|x-y| \le 10 \eta_1$ and $|x-y+\eta_2 v| \le 10\eta_1$ are negligible:}
  \medskip

  We note that the two parts are disjoint since $\eta_1 < \eta_2 / 100$ and $|v| \ge 1/2$. We have
  \begin{align*}
    & \eta_1^{-1} \eta_2^{-1} \Big|\int_{|x-y| \le \eta_1} f(x)f(y) E(x,y,\eta_1u,\eta_2v) \partial \varphi(u) \partial \varphi(v) \Big| \\
    & \lesssim \eta_1^{-1} \eta_2^{-1} \int_{|x-y| \le \eta_1} \frac{|x-y-\eta_1 u|^{\beta^2} |x-y+\eta_2 v|^{\beta^2}}{|x-y|^{\beta^2}|x-y-\eta_1 u+\eta_2 v|^{\beta^2}}.
  \end{align*}
  Note that $\frac{|x-y-\eta_1 u|^{\beta^2}}{|x-y|^{\beta^2}} = O(1)$ and similarly for $\frac{|x-y+\eta_2 v|^{\beta^2}}{|x-y-\eta_1 u + \eta_2 v|^{\beta^2}}$. Hence the whole integral is bounded by $\eta_1/\eta_2$. The case $|x-y+\eta_2v| \le 10 \eta_1$ can be handled in a similar fashion.

  \medskip

 \noindent  \textbf{The main part where $|x-y| > 10\eta_1$ and $|x-y+\eta_2 v| > 10\eta_1$:}
  \medskip

  We start by doing an integration by parts with respect to $u_1$. Due to our constraints there are no singularities or boundary terms, and we get
  \[\beta^2 \eta_2^{-1} \int f(x)f(y)E(x,y,\eta_1u,\eta_2v) (\partial_1 C(x-\eta u, y-\eta v) - \partial_1 C(x-\eta u, y)) \varphi(u) \partial \varphi(v).\]
  (We have omitted the constraints from under the integral sign for brevity.)
  Notice that
  \[C(x,y) - C(x-\eta_1 u, y) = O(\eta_1/|x-y|)\]
  and
  \[C(x-\eta_1 u, y - \eta_2 v) - C(x, y-\eta_2 v) = O(\eta_1/|x-y+\eta_2 v|),\]
  so that
  \[E(x,y,\eta_1 u, \eta_2 v) = 1 + O(\eta_1/|x-y|) + O(\eta_1/|x-y+\eta_2 v|).\]
  Let us consider the first $O$-term. Combined with the bound
  \[|\partial_1 C(x-\eta_1 u, y-\eta_2 v) - \partial_1 C(x-\eta_1 u, y)| \lesssim \frac{1}{|x-y+\eta_2v|} + \frac{1}{|x-y|}\]
  we can bound the integral in absolute value by a constant times
  \begin{align*}
    & \eta_2^{-1} \int_{|x-y| > 10\eta} \frac{\eta_1}{|x-y|^2} + \eta_2^{-1} \int_{\substack{|x-y| > 10\eta \\ |x-y+\eta_2v| > 10\eta}} \frac{\eta_1}{|x-y||x-y+\eta_2v|} \\
    & \lesssim \eta_1 \log(1/\eta_1) \eta_2^{-1} + \eta_1 \eta_2^{-1} \sqrt{\int_{|x-y| > 10\eta} \frac{1}{|x-y|^2}} \sqrt{\int_{|x-y+\eta_2 v| > 10\eta} \frac{1}{|x-y+\eta_2 v|^2}} \\
    & \lesssim \eta_1 \log(1/\eta_1) \eta_2^{-1}.
  \end{align*}
  The $O(\eta_1/|x-y+\eta_2 v|)$-term can be handled analogously.

  To finish the proof we note that the remaining term
  \[\beta^2 \eta_2^{-1} \int f(x)f(y) (\partial_1 C(x-\eta_1 u, y-\eta_2 v) - \partial_1 C(x-\eta_1 u, y)) \varphi(u) \partial \varphi(v)\]
  can be handled as in the end of the proof of Lemma~\ref{lemma:diagonalheta} and equals
  \[-\beta^2 \int f(x) \partial f(y + \eta_2 v) \partial_1 C(x-\eta_1 u, y) \varphi(u) \varphi(v).\]
  As before, the above tends to
  \[\beta^2 \int \partial f(x) \partial f(y) C(x,y) = \beta^2 \E \langle \partial \Gamma, f \rangle^2\]
  as $\eta_1,\eta_2 \to 0$.
\end{proof}

\end{document}